\documentclass{amsart}
\usepackage{amssymb}
\usepackage{hyperref}
\newtheorem{theorem}{Theorem}[section]
\newtheorem{definition}[theorem]{Definition}
\newtheorem{remark}{Remark}
\newtheorem{lemma}[theorem]{Lemma}
\numberwithin{equation}{section}

\DeclareMathOperator{\supp}{supp}

\DeclareFontFamily{U}{mathx}{\hyphenchar\font45}
\DeclareFontShape{U}{mathx}{m}{n}{
      <5> <6> <7> <8> <9> <10>
      <10.95> <12> <14.4> <17.28> <20.74> <24.88>
      mathx10
      }{}
\DeclareSymbolFont{mathx}{U}{mathx}{m}{n}
\DeclareFontSubstitution{U}{mathx}{m}{n}
\DeclareMathAccent{\widecheck}{0}{mathx}{"71}
\DeclareMathAccent{\wideparen}{0}{mathx}{"75}

\title[Length of sets under restricted projections onto lines]{Length of sets under restricted families of projections onto lines}
\author{Terence L.~J.~Harris}
\address{Department of Mathematics, Cornell University, Ithaca, NY 14853, USA}
\email{tlh236@cornell.edu}

\begin{document} 
\begin{abstract} Let $\gamma: I \to S^2$ be a $C^2$ curve with $\det(\gamma, \gamma', \gamma'')$ nonvanishing, and for each $\theta \in I$ let $\rho_{\theta}$ be orthogonal projection onto the span of $\gamma(\theta)$. It is shown that if $A \subseteq \mathbb{R}^3$ is a Borel set of Hausdorff dimension strictly greater than 1, then $\rho_{\theta}(A)$ has positive length for a.e.~$\theta \in I$. This answers a question raised by Käenmäki, Orponen and Venieri.   \end{abstract}
\maketitle
\section{Introduction} Let $S^2$ be the unit sphere in $\mathbb{R}^3$, let $\gamma: I \to S^2$ be a $C^2$ curve with $\det(\gamma, \gamma', \gamma'')$ nonvanishing on an interval $I$, and let $\rho_{\theta}$ be orthogonal projection onto the span of $\gamma(\theta)$, given by 
\[ \rho_{\theta}(x) = \langle x, \gamma(\theta) \rangle \gamma(\theta), \qquad x \in \mathbb{R}^3. \]
In \cite[Conjecture~1.6]{fasslerorponen14}, Fässler and Orponen conjectured that for any analytic set $A \subseteq \mathbb{R}^3$
\[ \dim \rho_{\theta} (A) = \min\{\dim A,1\}, \qquad  \text{a.e.~$ \theta \in I$,} \]
where $\dim B$ means the Hausdorff dimension of $B$. For general $C^2$ curves this was resolved by Pramanik, Yang and Zahl \cite{pramanik}, and independently by Gan, Guth and Maldague \cite{ggm}. The case where $\gamma$ is a non-great circle was resolved earlier by Käenmäki, Orponen and Venieri \cite{KOV}, who asked \cite[p.~4]{KOV} whether $\dim A > 1$ implies that $\rho_{\theta}(A)$ has positive length for a.e.~$\theta \in I$. This had been shown previously by Fässler and Orponen \cite[Theorem~1.9]{fasslerorponen14} in the special case where $A$ is a self-similar set without rotations. The following theorem resolves the general case. 
\begin{theorem} \label{maintheorem} If $A \subseteq \mathbb{R}^3$ is an analytic set with $\dim A > 1$, then
\[ \dim\left\{ \theta \in I : \mathcal{H}^1( \rho_{\theta}(A) ) = 0 \right\} \leq \frac{4- \dim A}{3}. \] \end{theorem}
Throughout, the symbol $\mathcal{H}^s$ will be used for the $s$-dimensional Hausdorff measure on Euclidean space. By Frostman's lemma, Theorem~\ref{maintheorem} will follow from Theorem~\ref{projmeasure} below. For the statement, some notation will be defined first. Given a Borel measure $\mu$ on $\mathbb{R}^3$ and $\alpha \geq 0$, define 
\[ c_{\alpha}(\mu) = \sup_{\substack{ x \in \mathbb{R}^3 \\ r >0 } } \frac{ \mu(B(x,r))}{r^{\alpha} }. \]
For each $\theta \in I$, the pushforward measure $\rho_{\theta\#}\mu$ is defined by 
\[ (\rho_{\theta\#} \mu)(E) = \mu\left(\rho_{\theta}^{-1}(E) \right), \]
for any Borel set $E$.
\begin{theorem} \label{projmeasure} If $\mu$ is a Borel measure on $\mathbb{R}^3$ with $c_{\alpha}(\mu) < \infty$ for some $\alpha>1$, then
\[ \dim\left\{ \theta \in I : \rho_{\theta \#} \mu \not \ll \mathcal{H}^1 \right\} \leq \frac{4 - \alpha}{3}. \] \end{theorem}
The notation $\rho_{\theta \#} \mu \not \ll \mathcal{H}^1$ means that $\rho_{\theta \#} \mu$ is not absolutely continuous with respect to $\mathcal{H}^1$. The proof of Theorem~\ref{projmeasure} is similar to the proof of Theorem~8 in \cite{GGGHMW}, which solved the analogous problem of projections onto planes. It uses a variant of the decomposition of a measure into ``good'' and ``bad'' parts, which originated in \cite{GIOW}. 

Given $s \geq 0$,  a set $A \subseteq \mathbb{R}^3$ is called an $s$-set if $A$ is $\mathcal{H}^s$-measurable with $0 < \mathcal{H}^s(A) < \infty$. Theorem~\ref{projmeasure} implies Theorem~\ref{maintheorem}, but it also implies the following slightly stronger version for $s$-sets.
\begin{theorem} \label{theoremgeneral} Suppose that $s>1$, and that $A \subseteq \mathbb{R}^3$ is an $s$-set. Let $\mu$ be the Borel measure defined by $\mu(F) = \mathcal{H}^s(F \cap A)$ for any Borel set $F \subseteq \mathbb{R}^3$, and let 
\[ E = \left\{ \theta \in I : \rho_{\theta \#} \mu \not\ll \mathcal{H}^1 \right\}. \]
 Then 
\[ \dim E \leq \frac{4-s}{3}, \]
and
\[ \mathcal{H}^1(\rho_{\theta}(B)) >0, \quad \text{ for all } \theta \in I \setminus E. \]
for any $\mathcal{H}^s$-measurable set $B \subseteq A$ with $\mathcal{H}^s(B)>0$.
 \end{theorem}
Theorem~\ref{theoremgeneral} is related to a lemma of Marstrand (see \cite[Lemma~13]{marstrand}), which states that for $s>1$, for any $s$-set $A$ in the plane, there is a measure zero set of exceptional directions, such that any $s$-set $B \subseteq A$ projects onto a set of positive Lebesgue measure outside of this set of exceptional directions.

Theorem~\ref{theoremgeneral} also implies Theorem~\ref{maintheorem}, since any analytic set of infinite $\mathcal{H}^s$ measure contains a closed set of positive finite $\mathcal{H}^s$ measure~\cite{davies}. 

\section{Proof of Theorem~\ref{projmeasure} and Theorem~\ref{theoremgeneral}}

Throughout this section, $\gamma: I \to S^2$ will be a fixed $C^2$ unit speed curve with $\det(\gamma, \gamma', \gamma'')$ nonvanishing on $I$, where $I$ is a compact interval. For $A \subseteq \mathbb{R}^3$, $m(A)$ will denote the Lebesgue measure of $A$.

\begin{definition} \label{decompdefn} Let
\[ \Lambda = \bigcup_{j \geq 1} \Lambda_j, \]
where each $\Lambda_j$ is a collection of boxes $\tau$ of dimensions $1 \times 2^{j/2} \times 2^j$, forming a finitely overlapping cover of the $\sim 1$-neighbourhood of the truncated light cone $\Gamma_j$ in the standard way, where 
\[ \Gamma_j = \left\{ t \gamma(\theta) : 2^{j-1} \leq \lvert t \rvert \leq 2^j, \quad \theta \in I \right\}. \] 
Each $\tau \in \Lambda_j$ has an angle $\theta_{\tau}$ such that the long axis of $\tau$ is parallel to $\gamma(\theta_{\tau} )$, the medium axis of $\tau$ is parallel to $\gamma'(\theta_{\tau})$, and the short axis of $\tau$ is parallel to $(\gamma \times \gamma')(\theta_{\tau})$. Let $\{\psi_{\tau}\}_{\tau \in \Lambda}$ be a smooth partition of unity subordinate to the cover $\left\{1.1\tau: \tau \in \Lambda\right\}$ of $\bigcup_{\tau \in \Lambda} \tau$, such that for any $\xi,v \in \mathbb{R}^3$, $t \in \mathbb{R}$ and any positive integer $n$,
\begin{multline} \label{derivconditiontau} \left\lvert \left( \frac{d}{dt} \right)^n \psi_{\tau}( \xi + t v ) \right\rvert \lesssim_n  \\
\left\lvert 2^{-j} \left\langle v, \gamma(\theta_{\tau}) \right\rangle \right\rvert^n + \left\lvert 2^{-j/2} \left\langle v, \gamma'(\theta_{\tau}) \right\rangle \right\rvert^n  + \left\lvert \left\langle v, (\gamma \times \gamma')(\theta_{\tau}) \right\rangle \right\rvert^n. \end{multline}

Fix $\widetilde{\delta}>0$. For each $\tau \in \Lambda$, let $\mathbb{T}_{\tau}$ be a collection of boxes (referred to as ``planks'') of dimensions $2^{j \widetilde{\delta}} \times 2^{j\left(\widetilde{\delta}-1/2\right)} \times 2^{j \left(\widetilde{\delta}-1\right)}$, forming a finitely overlapping cover of $\mathbb{R}^3$, such that for each $T \in \mathbb{T}_{\tau}$, the short axis of $T$ is parallel to $\gamma(\theta_{\tau})$, the medium axis is parallel to $\gamma'(\theta_{\tau})$, and the long axis is parallel to $(\gamma \times \gamma')(\theta_{\tau})$ (so that $T$ is ``dual'' to $\tau$). Let $\{\eta_T\}_{T \in \mathbb{T}_{\tau}}$ be a smooth partition of unity subordinate to the cover $\mathbb{T}_{\tau}$ of $\mathbb{R}^3$, such that for any $x,v \in \mathbb{R}^3$, any $t \in \mathbb{R}$ and any positive integer $n$, 
\begin{multline} \label{derivcondition} \left\lvert \left( \frac{d}{dt} \right)^n \eta_T( x + t v ) \right\rvert \lesssim_n  \\
\left\lvert 2^{j\left( 1- \widetilde{\delta} \right) } \left\langle v, \gamma(\theta_{\tau}) \right\rangle \right\rvert^n + \left\lvert 2^{j\left( 1/2- \widetilde{\delta} \right) } \left\langle v, \gamma'(\theta_{\tau}) \right\rangle \right\rvert^n  + \left\lvert 2^{-j \widetilde{\delta}} \left\langle v, (\gamma \times \gamma')(\theta_{\tau}) \right\rangle \right\rvert^n. \end{multline}

Given a finite complex Borel measure $\mu$ on $\mathbb{R}^3$, define
\[ M_T \mu = \eta_T \left( \mu \ast \widecheck{\psi_{\tau} } \right). \]  \end{definition}

For each $\theta \in I$, let $\pi_{\theta}$ be orthogonal projection onto the orthogonal complement of $(\gamma \times \gamma')(\theta)$. The following lemma is essentially a special case of Lemma~5 from \cite{GGGHMW}, but the proof will be included here for completeness. 

\begin{lemma} \label{IBP} There exists an $r>0$, depending only on $\gamma$, such that the following holds. Let $j \geq 1$ and let $\tau \in \Lambda_j$. If $\theta \in I$ is such that $2^{j\left( \widetilde{\delta} - 1/2 \right)} \leq \lvert \theta_{\tau} - \theta\rvert  \leq r$, then for any $T \in \mathbb{T}_{\tau}$, for any positive integer $N$ and for any $f \in L^1(\mathbb{R}^3)$,
\begin{equation} \label{l1inequality} \left\lVert \pi_{\theta \#} M_T f \right\rVert_{L^1(\mathcal{H}^2)} \leq C 2^{-j\widetilde{\delta} N } \left\lVert f \right\rVert_{L^1(\mathbb{R}^3) }, \end{equation}
where $C = C\left(N, \gamma, \widetilde{\delta}\right)$.  \end{lemma}
\begin{proof} For any $x \in (\gamma \times \gamma')(\theta)^{\perp}$, 
\begin{multline}\label{projdefn} (\pi_{\theta\#}M_Tf)(x) = \\
 \int_{\mathbb{R}^3} \int_{\tau}  f(y)  \psi_{\tau}(\xi) e^{-2\pi i \langle \xi, x-y\rangle} \left[ \int_{\mathbb{R}} \eta_T(x + t(\gamma \times \gamma')(\theta) ) e^{-2\pi it \langle \xi, (\gamma \times \gamma')(\theta) \rangle} \, dt \right] \, d\xi \, dy. \end{multline} 
Let $\xi \in \tau$. By integrating by parts $n$ times, and using \eqref{derivcondition},
\begin{multline} \label{firstline} \left\lvert \int_{\mathbb{R}} \eta_T(x + t(\gamma \times \gamma')(\theta) ) e^{-2\pi it \langle \xi, (\gamma \times \gamma')(\theta) \rangle} \, dt \right\rvert \lesssim  \int_{\mathbb{R}} \chi_T( x + t (\gamma \times \gamma')(\theta) ) \, dt \times \\
 \frac{2^{j \widetilde{\delta}}}{\left\lvert \langle (\gamma \times \gamma')(\theta), \xi \rangle \right\rvert^n} \bigg( 2^{j\left( 1- \widetilde{\delta} \right)}\left\lvert \left\langle (\gamma \times \gamma')(\theta), \gamma(\theta_{\tau} ) \right\rangle\right\rvert  + 2^{j\left( 1/2- \widetilde{\delta} \right)}\left\lvert \left\langle (\gamma \times \gamma')(\theta), \gamma'(\theta_{\tau} ) \right\rangle\right\rvert\\
 + 2^{-j \widetilde{\delta}}\left\lvert \left\langle (\gamma \times \gamma')(\theta), (\gamma \times \gamma')(\theta_{\tau} ) \right\rangle\right\rvert \bigg)^n. \end{multline}  
By the definition of $\tau$,
\begin{multline*} \langle (\gamma \times \gamma')(\theta), \xi \rangle  =\\
 \xi_1 \langle (\gamma \times \gamma')(\theta), \gamma(\theta_{\tau})  \rangle + \xi_2 \langle (\gamma \times \gamma')(\theta), \gamma'(\theta_{\tau})  \rangle + \xi_3 \langle (\gamma \times \gamma')(\theta), (\gamma \times \gamma')(\theta_{\tau})  \rangle, \end{multline*}
where $2^{j-2} \leq \lvert \xi_1 \rvert \leq 2^{j+2}$, $\lvert \xi_2 \rvert \leq 2^{j/2}$ and $\lvert \xi_3 \rvert \leq 1$. Let $\varepsilon := \lvert \theta_{\tau} - \theta\rvert$. By a second order Taylor approximation (using that $\gamma$ is $C^2$) and the scalar triple product formula,
\[ \left\lvert \xi_1 \langle (\gamma \times \gamma')(\theta), \gamma(\theta_{\tau})  \rangle \right\rvert \sim \varepsilon^2 2^j, \]
provided $r$ is small enough. Moreover
\[ \left\lvert \xi_2 \langle (\gamma \times \gamma')(\theta), \gamma'(\theta_{\tau})  \rangle\right\rvert \lesssim 2^{j/2} \varepsilon, \quad \left\lvert \xi_3 \langle (\gamma \times \gamma')(\theta), (\gamma \times \gamma')(\theta_{\tau}) \rangle \right\rvert \leq 1.  \]
Since $\varepsilon \geq 2^{j\left( \widetilde{\delta} - 1/2 \right)}$, it follows that 
\[ \left\lvert \langle (\gamma \times \gamma')(\theta), \xi \rangle \right\rvert \sim \varepsilon^2 2^j, \]
provided that $r$ is sufficiently small and $j$ is sufficiently large. Using second order Taylor approximation in a similar way gives that
\begin{multline*} 2^{j\left( 1- \widetilde{\delta} \right)}\left\lvert \left\langle (\gamma \times \gamma')(\theta), \gamma(\theta_{\tau} ) \right\rangle\right\rvert \\
 + 2^{j\left( 1/2- \widetilde{\delta} \right)}\left\lvert \left\langle (\gamma \times \gamma')(\theta), \gamma'(\theta_{\tau} ) \right\rangle\right\rvert + 2^{-j \widetilde{\delta}}\left\lvert \left\langle (\gamma \times \gamma')(\theta), (\gamma \times \gamma')(\theta_{\tau} ) \right\rangle\right\rvert \lesssim 2^{j\left(1-\widetilde{\delta} \right)}\varepsilon^2. \end{multline*} 
It follows that 
\[ \eqref{firstline} \lesssim_n 2^{-j \widetilde{\delta} (n-1) }  \int_{\mathbb{R}} \chi_T( x + t (\gamma \times \gamma')(\theta) ) \, dt. \] 
Substituting this into \eqref{projdefn} gives that 
\[ \left\lvert (\pi_{\theta\#}M_Tf)(x) \right\rvert \lesssim_n 2^{-j \widetilde{\delta} (n-1) } \left\lVert f\right\rVert_{L^1(\mathbb{R}^3)} m(\tau)  \int_{\mathbb{R}} \chi_T( x + t (\gamma \times \gamma')(\theta) ) \, dt, \]
for any $x \in (\gamma \times \gamma')(\theta)^{\perp}$. Integrating over $t \in \mathbb{R}$ and $x \in (\gamma \times \gamma')(\theta)^{\perp}$ gives \eqref{l1inequality}. \end{proof} 

The following lemma is essentially the same as Lemma~4 from \cite{GGGHMW}, but again the proof is included for completeness.
\begin{lemma} \label{MTf} Let $j \geq 1$ and let $\tau \in \Lambda_j$. For any finite compactly supported Borel measure $\mu$, 
\[ \left\lVert M_T \mu \right\rVert_{L^1(\mathbb{R}^3) } \leq 2^{3j \widetilde{\delta} } \mu(2T) + C_N 2^{-j \widetilde{\delta} N} \mu(\mathbb{R}^3), \]
for any positive integer $N$. \end{lemma}
\begin{proof} By definition,
\begin{align}\notag \left\lVert M_T \mu \right\rVert_{L^1(\mathbb{R}^3) } &= \int \eta_T(x) \left\lvert \int \widecheck{\psi_{\tau}}(x-y) \, d\mu(y) \right\rvert \, dx \\
\notag &\leq \int_{T} \int_{2T} \left\lvert \widecheck{\psi_{\tau}}(x-y) \right\rvert \, d\mu(y) \, dx \\
\label{starcircle} &\quad + \int_T \int_{\mathbb{R}^3 \setminus 2T} \left\lvert \widecheck{\psi_{\tau}}(x-y) \right\rvert \, d\mu(y) \, dx. \end{align}
The first integral satisfies 
\begin{equation} \label{star1} \int_{T} \int_{2T} \left\lvert \widecheck{\psi_{\tau}}(x-y) \right\rvert \, d\mu(y) \, dx  \leq 2^{3j \widetilde{\delta} } \mu(2T). \end{equation}
The second integral satisfies
\[ \int_T \int_{\mathbb{R}^3 \setminus 2T} \left\lvert \widecheck{\psi_{\tau}}(x-y) \right\rvert \, d\mu(y) \, dx \leq \mu(\mathbb{R}^3) \int_{\mathbb{R}^3 \setminus T_0} \left\lvert \widecheck{\psi_{\tau}} \right\rvert, \]  
where $T_0$ is the translate of the plank $T$ to the origin, parallel to $T$. Integrating by parts, and using \eqref{derivconditiontau}, gives that for any $k \geq 0$, for any $x \in \mathbb{R}^3 \setminus 2^k T_0$, 
\[ \left\lvert \widecheck{\psi_{\tau}}(x) \right\rvert \lesssim_N 2^{-kN-j \widetilde{\delta} N} m(\tau). \]
Summing a geometric series over $k \geq 0$ gives (with relabelled $N$)
\begin{equation} \label{star2} \int_T \int_{\mathbb{R}^3 \setminus 2T} \left\lvert \widecheck{\psi_{\tau}}(x-y) \right\rvert \, d\mu(y) \, dx \leq C_N\mu(\mathbb{R}^3)2^{-j \widetilde{\delta} N}.  \end{equation}
Putting \eqref{star1} and \eqref{star2} into \eqref{starcircle} finishes the proof. \end{proof}

For a function $f: X \to [0, +\infty]$ on a measure space $(X, \mathcal{A}, \mu)$, let $\int_* f \, d\mu$ denote the lower integral of $f$, defined by 
\[ \int_* f \, d\mu = \sup\left\{  \int g \, d\mu : \text{$g$ is simple and $\mathcal{A}$-measurable with $0 \leq g \leq f$} \right\}, \]
where ``simple'' means that $g$ takes finitely many values. In the application of Lemma~\ref{energylemma} below, the integrand will be measurable, so the use of the lower integral is not important and is only a technical convenience to avoid measurability issues. The definition of the lower integral is standard; see e.g.~\cite[p.~13]{mattila}. 

\begin{lemma}  \label{energylemma} Let $\beta \in [0,1]$, let $\alpha = 4-3\beta$, and let $\lambda$ be a Borel measure supported on $I$ with $c_{\beta}(\lambda) \leq 1$. Then for any $\epsilon >0$, there exists $\delta>0$ such that 
\begin{equation} \label{geometric} \int_{*}  (\rho_{\theta \#}\mu)\left( \bigcup_{D \in \mathbb{D}_{\theta}} D \right) \, d\lambda(\theta)  \leq C(\lambda, \delta, \epsilon, \gamma) R^{-\delta} \mu(\mathbb{R}^3), \end{equation}
for any $R \geq 1$, for any Borel measure $\mu$ on $B_3(0,1)$ with $c_{\alpha}(\mu) \leq 1$, and for any family of sets $\{ \mathbb{D}_{\theta} \}$, where each $\mathbb{D}_{\theta}$ is a disjoint set of intervals of diameter $2R^{-1}$ in the span of $\gamma(\theta)$, each with cardinality $\lvert\mathbb{D}_{\theta}\rvert \leq R^{1-\epsilon} \mu(\mathbb{R}^3)$. \end{lemma} 
\begin{remark}  Lemma~\ref{energylemma} roughly says that the pushforward $\rho_{\theta \#} \mu$ of a $(4-3\beta)$-dimensional measure $\mu$ satisfies a 1-dimensional Frostman condition on average. By taking $\beta =1$ (for example), this is enough to conclude that projections of 1-dimensional sets are a.e.~1-dimensional, but the main application of Lemma~\ref{energylemma} here will be to bound the $L^1$ norm of the ``bad'' part of the measure in the proof of Theorem~\ref{projmeasure}. The ``bad'' part of the measure corresponds to intervals of large $\rho_{\theta\#}\mu$-mass, which will automatically satisfy the cardinality assumption of Lemma~\ref{energylemma}. \end{remark}

\begin{proof}[Proof of Lemma~\ref{energylemma}] It may be assumed that $\alpha \leq 3$, since otherwise the lemma is trivial. Since the constant is allowed to depend on $\gamma$, it can be assumed that $\gamma$ is localised to a smaller interval on which Lemma~\ref{IBP} holds. 

Let $\phi_R$ be a non-negative bump function supported in $B(0, R^{-1})$ which integrates to 1, defined by 
\[ \phi_R(x) = R^3 \phi(Rx), \]
for some fixed non-negative bump function $\phi$ with support in $B(0,1)$ such that $\int \phi = 1$. For any $\theta$ and any $D \in \mathbb{D}_{\theta}$, the 1-Lipschitz property of orthogonal projections implies that
\[ \left[ \rho_{\theta\#}(\mu \ast \phi_R) \right](2D) \geq (\rho_{\theta\#}\mu)(D), \]
where $2D$ is the interval with the same centre as $D$, but twice the radius. Moreover, 
\[ c_{\alpha}( \mu \ast \phi_R) \lesssim c_{\alpha}(\mu), \] 
so it suffices to prove \eqref{geometric} with $\mu \ast \phi_R$ in place of $\mu$. To simplify notation the new measure will not be relabelled, but it will be assumed throughout that $\mu$ is a non-negative Schwartz function, and that
\begin{equation} \label{schwartzdecay} \left\lvert \widehat{\mu}(\xi) \right\vert \leq C_N (R/ \xi)^N, \quad \xi \in \mathbb{R}^3, \end{equation}
for any positive integer $N$, where $C_N$ is a constant depending only on $N$. 

Let $\epsilon_0$ be any positive real number which is strictly larger than the the infimum over all positive $\epsilon$ for which the conclusion of the lemma is true. It suffices to prove that the lemma holds for any $\epsilon > (2\epsilon_0)/3$, so let such an $\epsilon$ be given. Let $R \geq 1$ and choose a non-negative integer $J$ such that $2^J \sim R^{\epsilon/1000}$. Let $\varepsilon >0$ be such that $\varepsilon \ll \epsilon - \frac{2\epsilon_0}{3}$. Choose $\widetilde{\delta}>0$ such that $\widetilde{\delta} \ll \min\{\delta_{\epsilon_0}, \varepsilon\}$, where $\delta_{\epsilon_0}$ is a $\delta$ corresponding to $\epsilon_0$ that satisfies \eqref{geometric}. 

Define the ``bad'' part of $\mu$ by
\begin{equation} \label{baddefn} \mu_b = \sum_{j \geq J} \sum_{\tau \in \Lambda_j} \sum_{T \in \mathbb{T}_{\tau,b} } M_T \mu, \end{equation}
where, for each $\tau \in \Lambda_j$, the set of ``bad'' planks corresponding to $\tau$ is defined by 
\begin{equation} \label{badplankdefn} \mathbb{T}_{\tau, b} = \left\{ T \in \mathbb{T}_{\tau} : \mu(4T) \geq 2^{j(\epsilon_0 -1)} \right\}, \end{equation}
where $4T$ is a plank with the same centre as $T$, but scaled by a factor of 4. Define the ``good'' part of $\mu$ by 
\[ \mu_g = \mu-\mu_b. \] 
The Schwartz decay of $\mu$ implies that the sum in \eqref{baddefn} converges in the Schwartz space $\mathcal{S}(\mathbb{R}^3)$. This implies that $\mu_b$ and $\mu_g$ are Schwartz functions, and in particular they are finite complex measures. Pushforwards of complex measures are defined just as for positive measures. By\footnote{The particular Cauchy-Schwarz technique used here is from \cite{liu}.} Cauchy-Schwarz,
\begin{multline*} \int_{*} \left( \rho_{\theta \#} \mu \right)\left( \bigcup_{D \in \mathbb{D}_{\theta}} D \right) \, d\lambda(\theta)  \leq \int  \left\lVert \rho_{\theta \#} \mu_b \right\rVert_{L^1(\mathcal{H}^1) } \, d\lambda(\theta) \\ + \sup_{\theta \in I} \mathcal{H}^1\left(  \bigcup_{D \in \mathbb{D}_{\theta}} D \right)^{1/2} \left(\int  \left\lVert \rho_{\theta \#} \mu_g \right\rVert_{L^2(\mathcal{H}^1) }^2 \, d\lambda(\theta) \right)^{1/2}. \end{multline*}
The contribution from the ``bad'' part will be bounded first. By the triangle inequality,
\begin{align} \notag \int  \left\lVert \rho_{\theta \#} \mu_b \right\rVert_{L^1(\mathcal{H}^1) } \, d\lambda(\theta)  &\leq \sum_{j \geq J}   \int \sum_{\tau \in \Lambda_j} \sum_{T \in \mathbb{T}_{\tau,b} }\left\lVert \rho_{\theta \#} M_T \mu \right\rVert_{L^1(\mathcal{H}^1) } \, d\lambda(\theta) \\
\label{mainpart} &= \sum_{j \geq J}  \int  \sum_{\substack{\tau \in \Lambda_j: \\ \left\lvert \theta_{\tau} - \theta \right\rvert \leq 2^{j\left(\widetilde{\delta}-1/2\right)}}}  \sum_{T \in \mathbb{T}_{\tau,b} }\left\lVert \rho_{\theta \#} M_T \mu \right\rVert_{L^1(\mathcal{H}^1) } \, d\lambda(\theta)\\
\label{negligible} &\quad + \sum_{j \geq J}   \int  \sum_{\substack{\tau \in \Lambda_j: \\ \left\lvert \theta_{\tau} - \theta \right\rvert > 2^{j\left(\widetilde{\delta}-1/2\right)}}} \sum_{T \in \mathbb{T}_{\tau,b} }\left\lVert \rho_{\theta \#} M_T \mu \right\rVert_{L^1(\mathcal{H}^1) } \, d\lambda(\theta). \end{align}
Let $\pi_{\theta}$ be orthogonal projection onto $(\gamma \times \gamma')(\theta)^{\perp}$. By the inequality
\[ \left\lVert \rho_{\theta \#} f \right\rVert_{L^1(\mathcal{H}^1)} \leq \left\lVert \pi_{\theta \#} f \right\rVert_{L^1(\mathcal{H}^2)}, \]
followed by Lemma~\ref{IBP}, 
\[ \eqref{negligible}  \lesssim 2^{-J}\mu(\mathbb{R}^3) \sim R^{-\epsilon/1000}\mu(\mathbb{R}^3).\] 
By the inequality
\[ \left\lVert \rho_{\theta \#} f \right\rVert_{L^1(\mathcal{H}^1)} \leq \left\lVert f \right\rVert_{L^1(\mathbb{R}^3)}, \]
followed by Lemma~\ref{MTf},
\begin{align*} \eqref{mainpart} &\leq  \sum_{j \geq J}  \int \sum_{\substack{\tau \in \Lambda_j: \\ \left\lvert \theta_{\tau} - \theta \right\rvert \leq 2^{j\left(\widetilde{\delta}-1/2\right)}}}  \sum_{T \in \mathbb{T}_{\tau,b} }\left\lVert  M_T \mu \right\rVert_{L^1(\mathbb{R}^3) } \, d\lambda(\theta) \\
&\lesssim 2^{-J}\mu(\mathbb{R}^3) + \sum_{j \geq J} 2^{3j \widetilde{\delta}} \int  \sum_{\substack{\tau \in \Lambda_j: \\ \left\lvert \theta_{\tau} - \theta \right\rvert \leq 2^{j\left(\widetilde{\delta}-1/2\right)}}}  \sum_{T \in \mathbb{T}_{\tau,b} }\mu(2T) \, d\lambda(\theta), \end{align*}
The non-tail term satisfies 
\begin{equation} \label{nontail} \sum_{j \geq J} 2^{3j \widetilde{\delta}} \int  \sum_{\substack{\tau \in \Lambda_j: \\ \left\lvert \theta_{\tau} - \theta \right\rvert \leq 2^{j\left(\widetilde{\delta}-1/2\right)}}}  \sum_{T \in \mathbb{T}_{\tau,b} }\mu(2T) \, d\lambda(\theta) \lesssim  \sum_{j \geq J} 2^{10j \widetilde{\delta}}\int \mu(B_j(\theta) ) \, d\lambda(\theta),  \end{equation}
where, for each $\theta \in I$ and each $j$, 
\[ B_j(\theta) =  \bigcup_{\substack{\tau \in \Lambda_j: \\ \left\lvert \theta_{\tau} - \theta \right\rvert \leq 2^{j\left(\widetilde{\delta}-1/2\right)}}}  \bigcup_{T \in \mathbb{T}_{\tau,b} } 2T. \]
The inequality \eqref{nontail} used that for each $j$ and each $\theta \in I$, there are $\lesssim 2^{j \widetilde{\delta}}$ sets $\tau \in \Lambda_j$ with the property that $\left\lvert \theta_{\tau} - \theta \right\rvert \leq 2^{j\left(\widetilde{\delta}-1/2\right)}$, which means each of the planks $2T$ in the union defining $B_j(\theta)$ intersects $\lesssim 2^{j \widetilde{\delta}}$ of the others. For each $T$ in the union defining $B_j(\theta)$, the set $(4T) \cap B(0,1)$ is contained in a plank $T_{\theta}$ of dimensions
\[ \sim 2^{j\left( 2 \widetilde{\delta} -1 \right)} \times 2^{j\left( \widetilde{\delta} - 1/2\right) } \times 1, \]
with short direction parallel to $\gamma(\theta)$, medium direction parallel to $\gamma'(\theta)$, and long direction parallel to $(\gamma \times \gamma')(\theta)$, where the implicit constant depends only on $\gamma$; this follows from the second order Taylor approximation for $\gamma$. Therefore, the intervals in the set
\[ \left\{ \rho_{\theta}(T_{\theta}) : T \in \mathbb{T}_{\tau,b}, \quad \tau \in \Lambda_j,  \quad \left\lvert \theta_{\tau} - \theta \right\rvert \leq 2^{j\left(\widetilde{\delta}-1/2\right)} \right\}, \]
all have length $\sim 2^{j\left( 2 \widetilde{\delta} -1 \right)}$, and form a cover of $\rho_{\theta}(B_j(\theta) \cap B(0,1))$. By the Vitali covering lemma, there is a disjoint subcollection 
\[ \{ \rho_{\theta}(T_{\theta}) : T \in \mathcal{B}_{\theta} \}, \]
indexed by some set $\mathcal{B}_{\theta}$, such that 
\[ \{ 3\rho_{\theta}(T_{\theta}) : T \in \mathcal{B}_{\theta} \}, \]
is a cover of $\rho_{\theta}(B_j(\theta) \cap B(0,1))$. The set $\mathcal{B}_{\theta}$ has cardinality $\left\lvert \mathcal{B}_{\theta}\right\rvert \leq \mu(\mathbb{R}^3)2^{j(1-\epsilon_0)}$; by disjointness and the definition of the ``bad'' planks (see \eqref{badplankdefn}). Since the conclusion of the lemma holds for $\epsilon_0$, it follows that for each $j \geq J$,
\[ \int \mu(B_j(\theta) ) \, d\lambda(\theta) \leq \int \left( \rho_{\theta\#} \mu \right) \left( \bigcup_{T_{\theta} \in \mathcal{B}_{\theta} } 3 \rho_{\theta}(T_{\theta} ) \right) \, d\lambda(\theta) \lesssim 2^{j \left(-\delta_{\epsilon_0}/2 + 10 \widetilde{\delta} \right)} \mu(\mathbb{R}^3). \]
The set $B_j(\theta)$ is piecewise constant in $\theta$ over a partition of $I$ into Borel sets, so it may be assumed that the integrands above are Borel measurable. Since $\widetilde{\delta} \ll \delta_{\epsilon_0}$, the inequality above yields
\[ \eqref{mainpart} \lesssim \mu(\mathbb{R}^3)2^{-(J \delta_{\epsilon_0})/100} \sim \mu(\mathbb{R}^3)R^{-(\epsilon \delta_{\epsilon_0})/10^5}. \]

It remains to bound the contribution from $\mu_g$.  By the assumptions in the lemma, 
\[  \sup_{\theta \in I} \mathcal{H}^1\left(  \bigcup_{D \in \mathbb{D}_{\theta}} D \right) \lesssim  R^{-\epsilon} \mu(\mathbb{R}^3). \]  
Since $\varepsilon \ll \epsilon - (2\epsilon_0)/3$, it suffices to prove that
\[ \int \left\lVert \rho_{\theta \#} \mu_g \right\rVert_{L^2(\mathcal{H}^1) }^2 \, d\lambda(\theta) \lesssim \max\left\{R^{2\epsilon_0/3 + 1000 \varepsilon}, R^{\epsilon/2} \right\} \mu(\mathbb{R}^3), \]
By Plancherel's theorem in 1 dimension,
\[ \int  \left\lVert \rho_{\theta \#} \mu_g \right\rVert_{L^2(\mathcal{H}^1) }^2\, d\lambda(\theta) = \int \int_{\mathbb{R}} \left\lvert \widehat{ \mu_g} ( t \gamma(\theta) ) \right\rvert^2 \, dt \, d\lambda(\theta). \]
To formally prove this identity, one approach is to rotate $\gamma(\theta)$ to $(1,0,0)$ (using that $\mathcal{H}^1$ is a rotation invariant measure on $\mathbb{R}^3$), and apply Plancherel's theorem on $\mathbb{R}$. By symmetry, by summing a geometric series, and by the rapid decay of $\widehat{\mu}$ outside $B(0,R)$ (see \eqref{schwartzdecay}), it will suffice to bound
\[ \int \int_{2^{j-1}}^{2^j}\left\lvert \widehat{\mu_g}\left( t \gamma(\theta) \right) \right\rvert^2 \, dt \, d\lambda(\theta),   \]
for any $j \geq 2J$ with $2^j \leq R^{1+\widetilde{\delta}}$; the contribution from the small frequencies can be bounded trivially by the definition of $J$. For each $\tau \in \Lambda$, define the set of ``good'' planks corresponding to $\tau$ by
\[ \mathbb{T}_{\tau,g} = \mathbb{T}_{\tau} \setminus \mathbb{T}_{\tau, b}. \]
 Then 
\begin{multline*} \int \int_{2^{j-1}}^{2^j}\left\lvert \widehat{\mu_g}\left( t \gamma(\theta) \right) \right\rvert^2 \, dt \, d\lambda(\theta) \\
\lesssim \int  \int_{2^{j-1}}^{2^j}\left\lvert \sum_{\tau \in \bigcup_{\left\lvert j' -j\right\rvert \leq 2} \Lambda_{j'}} \sum_{T \in \mathbb{T}_{\tau,g} } \widehat{M_T\mu}\left( t \gamma(\theta) \right) \right\rvert^2 \, dt \, d\lambda(\theta) +  2^{-j}\mu(\mathbb{R}^3). \end{multline*}
Since the $\tau$'s are finitely overlapping, 
\begin{multline} \label{intermezzo} \int \int_{2^{j-1}}^{2^j}\left\lvert \sum_{\tau \in \bigcup_{\left\lvert j' -j\right\rvert \leq 2} \Lambda_{j'}} \sum_{T \in \mathbb{T}_{\tau,g} } \widehat{M_T\mu}\left( t \gamma(\theta) \right) \right\rvert^2 \, dt \, d\lambda(\theta) \lesssim 2^{-j}\mu(\mathbb{R}^3)  \\
+ \sum_{\tau \in \bigcup_{\left\lvert j' -j\right\rvert \leq 2} \Lambda_{j'} } \int \int_{2^{j-1}}^{2^j}\left\lvert \sum_{T \in \mathbb{T}_{\tau, g} } \widehat{M_T\mu}\left( t \gamma(\theta) \right) \right\rvert^2 \, dt \, d\lambda(\theta). \end{multline}
The uncertainty principle implies that each of the integrals in the right-hand side of \eqref{intermezzo} is bounded by $2^{-j\beta}$ times the integral of the same function over $\mathbb{R}^3$. More precisely, for each $\tau \in\bigcup_{\left\lvert j' -j\right\rvert \leq 2} \Lambda_{j'}$, the contribution from the planks in the above sum with $T \cap B\left( 0,2^{10j \widetilde{\delta}}\right) = \emptyset$ is negligible since $\mu$ is supported in $B(0,1)$. The remaining sum 
\[ g_{\tau} = \sum_{T \in \mathbb{T}_{\tau, g} : T \cap B\left( 0,2^{10j \widetilde{\delta}}\right) \neq \emptyset   } M_T \mu, \]
is equal to $g_{\tau} \varphi$ where $\varphi$ is a smooth bump function on $B\left( 0,  2^{ 1 + 10j \widetilde{\delta}}\right)$ obtained by rescaling a bump function on the unit ball, and therefore by the Cauchy-Schwarz inequality,
\begin{equation} \label{essentiallyconstant} \left\lvert \widehat{g_{\tau} } \right\rvert^2 \lesssim \left\lvert \widehat{g_{\tau} } \right\rvert^2 \ast \zeta, \end{equation}
where 
\[ \zeta(\xi) = \frac{2^{30j\delta}}{1+  \lvert 2^{10j \delta} \xi \rvert^{100}}, \qquad \xi \in \mathbb{R}^3. \]
The function in the right-hand side of \eqref{essentiallyconstant} is essentially constant on balls of radius $2^{-10j \widetilde{\delta}}$ (as it inherits this property from $\zeta$), so by discretising the integral in \eqref{intermezzo} into balls of radius $2^{-10j \widetilde{\delta}}$ and using the condition $c_{\beta}(\lambda) \leq 1$, 
\begin{multline*} \int \int_{2^{j-1}}^{2^j}\left\lvert \sum_{T \in \mathbb{T}_{\tau, g} } \widehat{M_T\mu}\left( t \gamma(\theta) \right) \right\rvert^2 \, dt \, d\lambda(\theta) \lesssim \\
2^{j\left(100\widetilde{\delta}-\beta\right)} \int_{\mathbb{R}^3} \left\lvert \sum_{T \in \mathbb{T}_{\tau, g} : T \cap B\left( 0,2^{10j \widetilde{\delta}}\right) \neq \emptyset   } \widehat{M_T\mu} \right\rvert^2 \, d\xi +  2^{-j} \mu(\mathbb{R}^3), \end{multline*} 
for each $\tau \in\bigcup_{\left\lvert j' -j\right\rvert \leq 2} \Lambda_{j'}$. Let 
\[ \mathbb{T}_{j, g} = \bigcup_{\tau \in \bigcup_{\left\lvert j' -j\right\rvert \leq 2} \Lambda_{j'}} \left\{ T \in \mathbb{T}_{\tau, g}: T \cap B\left( 0,2^{10j \widetilde{\delta}}\right) \neq \emptyset   \right\}. \]
By Plancherel's theorem in $\mathbb{R}^3$ and the finite overlapping property of the $T$'s, it suffices to prove that for any $j \geq 2J$,
\begin{equation} \label{enough} \sum_{T \in \mathbb{T}_{j, g}} \int_{\mathbb{R}^3} \left\lvert M_T \mu\right\rvert^2 \, dx \lesssim 2^{j\left( \beta + \frac{2 \epsilon_0}{3} + 100 \varepsilon \right) } \mu(\mathbb{R}^3). \end{equation}
From the definition $M_T \mu = \eta_T \left( \mu \ast \widecheck{\psi_{\tau(T)} } \right)$ and by Fubini, the left-hand side of the above is equal to
\[  \int  \sum_{T \in \mathbb{T}_{j, g}} \left[\eta_T M_T \mu \right] \ast \widecheck{\psi_{\tau(T)}} \, d\mu. \]
If $f_T := \left[\eta_T M_T \mu \right] \ast \widecheck{\psi_{\tau(T)}}$, then by the Cauchy-Schwarz inequality with respect to the measure $\mu$, the square of the above is bounded by
\[  \int \left\lvert \sum_{T \in \mathbb{T}_{j, g}} f_T \right\rvert^2 \, d\mu \cdot \mu(\mathbb{R}^3). \]
By the uncertainty principle,
\[  \int  \left\lvert \sum_{T \in \mathbb{T}_{j, g}} f_T \right\rvert^2 \, d\mu \lesssim \int  \left\lvert \sum_{T \in \mathbb{T}_{j, g}} f_T \right\rvert^2 \, d\mu_j, \]
where $\mu_j = \mu \ast \phi_j$ and $\phi_j(x) = \frac{2^{3j}}{1+\left\lvert 2^j x\right\rvert^{N}}$ , where $N \sim 1000/ \widetilde{\delta}^2$. By dyadic pigeonholing, there exists a collection $\mathbb{W}$ of planks $T \in \mathbb{T}_{j,g}$ with $\left\lVert f_T\right\rVert_p$ constant over $T \in \mathbb{W}$ up to a factor of 2, and a union $Y$ of disjoint $2^{-j}$-balls $Q$ such that each $Q$ intersects $\sim M$ planks $2T \in \mathbb{W}$ for some dyadic number $M$, and such that
\[  \int \left\lvert \sum_{T \in \mathbb{T}_{j, g}} f_T \right\rvert^2 \, d\mu_j \lesssim j^{10} \int_{Y} \left\lvert \sum_{T \in \mathbb{W}} f_T \right\rvert^2 \, d\mu_j + 2^{-j} \mu(\mathbb{R}^3)^2, \]
Let $p=6$. By Hölder's inequality with respect to the Lebesgue measure,
\begin{equation} \label{pause} \int_{Y} \left\lvert \sum_{T \in \mathbb{W}} f_T \right\rvert^2 \, d\mu_j \leq \left\lVert  \sum_{T \in \mathbb{W}} f_T \right\rVert_{L^p(Y)}^2 \left( \int_Y \mu_j^{\frac{p}{p-2} }\right)^{\frac{p-2}{p} }, \end{equation}
By the dimension condition $c_{\alpha}(\mu) \leq 1$ on $\mu$, the definition of $Y$, and the definition of the ``good'' planks,
\begin{align} \label{muestimate} \int_Y \mu_j^{\frac{p}{p-2} } &\lesssim 2^{\frac{ 2j(3-\alpha)}{p-2} } \int_Y \mu_j \\
\notag &\leq 2^{\frac{ 2j(3-\alpha)}{p-2} } \sum_{Q \subseteq Y} \int_Q \mu_j \\
\notag &\lesssim  \left( \frac{1}{M} \right) \cdot 2^{\frac{ 2j(3-\alpha)}{p-2} } \sum_{Q \subseteq Y} \sum_{\substack{T \in \mathbb{W} \\ 2T \cap Q \neq \emptyset } } \int_{Q \cap 3T} \mu_j \\
\notag &\lesssim  \left( \frac{1}{M} \right) \cdot 2^{\frac{ 2j(3-\alpha)}{p-2} }  \sum_{T \in \mathbb{W}} \int_{3T} \mu_j \\
\notag &\lesssim  \left( \frac{1}{M} \right) \cdot 2^{\frac{ 2j(3-\alpha)}{p-2} }  \sum_{T \in \mathbb{W}} \mu(4T) + 2^{-100j} \\
\notag &\lesssim 2^{\frac{2j(3-\alpha)}{p-2} + j\left(\epsilon_0 - 1\right)} \left( \frac{ \left\lvert \mathbb{W} \right\rvert }{M} \right). \end{align}
This bounds the second factor in \eqref{pause}, so it remains to bound the first factor. 

By rescaling by $2^j$, applying the refined decoupling inequality (see Theorem~\ref{refineddecouplingtheorem} of the appendix), and then rescaling back,
\[ \left\lVert \sum_{T \in \mathbb{W}} f_T \right\rVert_{L^p(Y)} \lesssim 2^{j \varepsilon} \left( \frac{M}{\left\lvert \mathbb{W} \right\rvert} \right)^{\frac{1}{2} - \frac{1}{p} } \left( \sum_{T \in \mathbb{W} } \left\lVert f_T\right\rVert_p^2 \right)^{1/2}. \] 
Recall that $f_T = \left[\eta_T M_T \mu \right] \ast \widecheck{\psi_{\tau(T)}}$. By applying the Hausdorff-Young inequality, then Hölder's inequality, and then Plancherel's theorem,
\[ \left\lVert f_T\right\rVert_p \lesssim  \left\lVert M_T \mu\right\rVert_2 2^{\frac{3j}{2} \left( \frac{1}{2} - \frac{1}{p}\right)}. \]
Hence 
\begin{equation} \label{refineddecoupling} \left\lVert \sum_{T \in \mathbb{W}} f_T \right\rVert_{L^p(Y)} \lesssim 2^{\frac{3j}{2} \left( \frac{1}{2} - \frac{1}{p} + \varepsilon \right)} \left( \frac{M}{\left\lvert \mathbb{W} \right\rvert} \right)^{\frac{1}{2} - \frac{1}{p} } \left( \sum_{T \in \mathbb{W} } \left\lVert M_T \mu\right\rVert_2^2 \right)^{1/2}. \end{equation}
Putting \eqref{refineddecoupling} and \eqref{muestimate} into \eqref{pause} gives
\[ \sum_{T \in \mathbb{T}_{j, g}} \int_{\mathbb{R}^3} \left\lvert M_T \mu\right\rvert^2 \lesssim 2^{j \left[ \frac{5-2\alpha}{2p} + \frac{1}{4} +\frac{\epsilon_0(p-2)}{2p} + \frac{3\varepsilon}{2}\right]} \left( \sum_{T \in \mathbb{T}_{j, g}} \int_{\mathbb{R}^3} \left\lvert M_T \mu\right\rvert^2 \right)^{1/2} \mu(\mathbb{R}^3)^{1/2}. \]
By cancelling the common factor, this gives
\[ \sum_{T \in \mathbb{T}_{j, g}} \int_{\mathbb{R}^3} \left\lvert M_T \mu\right\rvert^2 \lesssim  2^{j \left[ \frac{5-2\alpha}{p} + \frac{1}{2} +\frac{\epsilon_0(p-2)}{p} + 3\varepsilon  \right]}\mu(\mathbb{R}^3). \]
Since $p=6$ and $\alpha = 4- 3\beta$, this simplifies to
\[ \sum_{T \in \mathbb{T}_{j, g}} \int_{\mathbb{R}^3} \left\lvert M_T \mu\right\rvert^2 \lesssim  2^{j\left( \beta+ \frac{2\epsilon_0}{3} + 3\varepsilon \right)}\mu(\mathbb{R}^3), \]
which verifies \eqref{enough}, and as explained above, this proves the lemma.  
\end{proof} 

The proof of Theorem~\ref{projmeasure} will be similar to the proof of the lemma.
\begin{proof}[Proof of Theorem~\ref{projmeasure}] It will first be shown that the set
\[ \left\{ \theta \in I: \rho_{\theta \#} \mu \not\ll \mathcal{H}^1 \right\} \]
is Borel measurable. It may be assumed that $\mu$ is compactly supported, since if $\mu_k$ is the restriction of $\mu$ to $B(0,k)$, then 
\[ \left\{ \theta \in I: \rho_{\theta \#} \mu \not\ll \mathcal{H}^1 \right\} = \bigcup_{k=1}^{\infty} \left\{ \theta \in I: \rho_{\theta \#} \mu_k \not\ll \mathcal{H}^1 \right\}. \]
For any positive integer $n$, the function
\[ (\theta,t) \mapsto \left(\rho_{\theta \#} \mu\right)(B(t\gamma(\theta), 1/n ) ), \]
is lower semicontinuous on $I \times \mathbb{R}$, and is therefore Borel measurable from $I \times \mathbb{R}$ to $[0, +\infty]$ (here $B(x,r)$ denotes the \emph{open} ball or interval of radius $r$ around $x$). It follows that the function 
\[ (\theta,t) \mapsto \limsup_{n \to \infty}  n\left(\rho_{\theta \#} \mu\right)(B(t\gamma(\theta), 1/n ) ) \]
is Borel measurable from $I \times \mathbb{R}$ to $[0, +\infty]$. By Tonelli's theorem (see e.g.~\cite[Theorem~1.7.15]{tao}), the function 
\[ \theta \mapsto \int_{\mathbb{R}} \limsup_{n \to \infty}  n\left(\rho_{\theta \#} \mu\right)(B(t\gamma(\theta), 1/n ) ) \, dt \]
is Borel measurable from $I$ to $[0, +\infty]$. It follows that the set
\[ \left\{ \theta \in I : \int_{\mathbb{R}} \limsup_{n \to \infty}  n\left(\rho_{\theta \#} \mu\right)(B(t\gamma(\theta), 1/n ) ) \, dt  < \mu(\mathbb{R}^3) \right\}, \]
is a Borel measurable subset of $I$, and this set is equal to $\left\{ \theta \in I: \rho_{\theta \#} \mu \not\ll \mathcal{H}^1 \right\}$ by the Lebesgue differentiation theorem (see e.g.~\cite[Theorem~3.22]{folland}).

As in the proof of the lemma, it may be assumed that $\gamma$ is localised to a small interval on which Lemma~\ref{IBP} holds.  It may also be assumed that $\alpha \leq 3$, $c_{\alpha}(\mu) \leq 1$ and (by countable stability of the Hausdorff dimension) that $\mu$ has support in the unit ball. Let $\beta$ be such that $0 \leq \beta < (4-\alpha)/3$. Let $\lambda$ be a Borel measure supported on $I$ with $c_{\beta}(\lambda) \leq 1$. Let $\epsilon>0$ be such that $\epsilon \ll \frac{4-\alpha}{3} -\beta$. Choose $\widetilde{\delta}>0$ such that $\widetilde{\delta} \ll \min\left\{ \epsilon, \delta_{\epsilon} \right\}$, where $\delta_{\epsilon}$ is an exponent corresponding to $\epsilon$ from Lemma~\ref{energylemma}. Using Definition~\ref{decompdefn}, define $\mu_b$ by 
\begin{equation} \label{mubaddefn2}  \mu_b = \sum_{j \geq 1} \sum_{\tau \in \Lambda_j} \sum_{T \in \mathbb{T}_{\tau,b} } M_T \mu, \end{equation}
where, for each $j \geq 1$ and $\tau \in \Lambda_j$, the set of ``bad'' planks corresponding to $\tau$ is defined by 
\[ \mathbb{T}_{\tau, b} = \left\{ T \in \mathbb{T}_{\tau} : \mu(4T) \geq 2^{j(\epsilon -1)} \right\}. \]
Since the frequencies are no longer localised, the sum in \eqref{mubaddefn2} only converges a~priori in the space of tempered distributions, but the individual functions $M_T \mu$ are smooth  and compactly supported. For $\lambda$-a.e.~$\theta \in I$,
\begin{equation} \label{pushbaddefn} \rho_{\theta \#}  \mu_b  := \sum_{j \geq 1} \sum_{\tau \in \Lambda_j} \sum_{T \in \mathbb{T}_{\tau,b} } \rho_{\theta \#} M_T \mu, \end{equation}
where, for $\lambda$-a.e.~$\theta \in I$, the series will be shown to be absolutely convergent in $L^1(\mathcal{H}^1)$. Since $L^1$ is always a Banach space, any absolutely convergent series of $L^1$ functions is convergent in $L^1$, so the $\lambda$-a.e.~absolute convergence of \eqref{pushbaddefn} in $L^1(\mathcal{H}^1)$ will imply that $\rho_{\theta\#} \mu_b \in L^1(\mathcal{H}^1)$ for $\lambda$-a.e.~$\theta \in I$, and will imply that the series is well-defined as an $L^1(\mathcal{H}^1)$ limit. Define 
\[ \rho_{\theta\#} \mu_g = \rho_{\theta\#} \mu - \rho_{\theta \#} \mu_b, \]
for each $\lambda \in I$ such that the sum defining $\rho_{\theta\#}\mu_b$ converges in $L^1(\mathcal{H}^1)$ (which will include $\lambda$-a.e.~$\theta \in I$). It will be shown that 
\[ \rho_{\theta \#} \mu_g \in L^2(\mathcal{H}^1), \]
for $\lambda$-a.e.~$\theta \in I$. Together with $\rho_{\theta\#} \mu_b \in L^1(\mathcal{H}^1)$, this will imply that $\rho_{\theta \#}\mu \in L^1(\mathcal{H}^1)$ (or equivalently $\rho_{\theta \#}\mu \ll \mathcal{H}^1$) for $\lambda$-a.e.~$\theta \in I$.

It will first be shown that
\begin{equation} \label{pause333} \int \sum_{j \geq 1} \sum_{\tau \in \Lambda_j} \sum_{T \in \mathbb{T}_{\tau,b} } \left\lVert \rho_{\theta \#}  M_T \mu \right\rVert_{L^1(\mathcal{H}^1 ) } \, d\lambda(\theta) < \infty. \end{equation}
The proof of this is similar to the proof of Lemma~\ref{energylemma}, but some of the details will be included for readability. The left-hand side of \eqref{pause333} can be written as
\begin{align}  \notag &\sum_{j \geq 1}   \int \sum_{\tau \in \Lambda_j} \sum_{T \in \mathbb{T}_{\tau,b} }\left\lVert \rho_{\theta \#}  M_T \mu\right\rVert_{L^1(\mathcal{H}^1) } \, d\lambda(\theta)  \\
\label{mainpart2} &\quad = \sum_{j \geq 1}  \int  \sum_{\substack{\tau \in \Lambda_j: \\ \left\lvert \theta_{\tau} - \theta \right\rvert \leq 2^{j\left(\widetilde{\delta}-1/2\right)}}}  \sum_{T \in \mathbb{T}_{\tau,b} }\left\lVert \rho_{\theta \#} M_T \mu\right\rVert_{L^1(\mathcal{H}^1) } \, d\lambda(\theta)\\
\label{negligible2} &\qquad + \sum_{j \geq 1}   \int  \sum_{\substack{\tau \in \Lambda_j: \\ \left\lvert \theta_{\tau} - \theta \right\rvert > 2^{j\left(\widetilde{\delta}-1/2\right)}}} \sum_{T \in \mathbb{T}_{\tau,b} }\left\lVert \rho_{\theta \#} M_T \mu \right\rVert_{L^1(\mathcal{H}^1) } \, d\lambda(\theta). \end{align}
By Lemma~\ref{IBP},
\[ \eqref{negligible2}  \lesssim \mu(\mathbb{R}^3).\] 
By Lemma~\ref{MTf},
\begin{align*} \eqref{mainpart2} &\leq  \sum_{j \geq 1}  \int \sum_{\substack{\tau \in \Lambda_j: \\ \left\lvert \theta_{\tau} - \theta \right\rvert \leq 2^{j\left(\widetilde{\delta}-1/2\right)}}}  \sum_{T \in \mathbb{T}_{\tau,b} }\left\lVert  M_T \mu \right\rVert_{L^1(\mathbb{R}^3) } \, d\lambda(\theta) \\
&\lesssim \mu(\mathbb{R}^3) + \sum_{j \geq 1} 2^{3j \widetilde{\delta}} \int  \sum_{\substack{\tau \in \Lambda_j: \\ \left\lvert \theta_{\tau} - \theta \right\rvert \leq 2^{j\left(\widetilde{\delta}-1/2\right)}}}  \sum_{T \in \mathbb{T}_{\tau,b} }\mu(2T) \, d\lambda(\theta). \end{align*}
As in the proof of Lemma~\ref{energylemma}, the non-tail term satisfies 
\[ \sum_{j \geq 1} 2^{3j \widetilde{\delta}} \int  \sum_{\substack{\tau \in \Lambda_j: \\ \left\lvert \theta_{\tau} - \theta \right\rvert \leq 2^{j\left(\widetilde{\delta}-1/2\right)}}}  \sum_{T \in \mathbb{T}_{\tau,b} }\mu(2T) \, d\lambda(\theta) \lesssim  \sum_{j \geq 1} 2^{10j \widetilde{\delta}}\int \mu(B_j(\theta) ) \, d\lambda(\theta),  \]
where, for each $\theta \in I$ and each $j$, 
\[ B_j(\theta) =  \bigcup_{\substack{\tau \in \Lambda_j: \\ \left\lvert \theta_{\tau} - \theta \right\rvert \leq 2^{j\left(\widetilde{\delta}-1/2\right)}}}  \bigcup_{T \in \mathbb{T}_{\tau,b} } 2T. \]
For each $T$ in the union defining $B_j(\theta)$, the set $(4T) \cap B(0,1)$ is contained in a plank $T_{\theta}$ of dimensions
\[ \sim 2^{j\left( 2 \widetilde{\delta} -1 \right)} \times 2^{j\left( \widetilde{\delta} - 1/2\right) } \times 1, \]
with short direction parallel to $\gamma(\theta)$, medium direction parallel to $\gamma'(\theta)$, and long direction parallel to $(\gamma \times \gamma')(\theta)$. The intervals in the set
\[ \left\{ \rho_{\theta}(T_{\theta}) : T \in \mathbb{T}_{\tau,b}, \quad \tau \in \Lambda_j,  \quad \left\lvert \theta_{\tau} - \theta \right\rvert \leq 2^{j\left(\widetilde{\delta}-1/2\right)} \right\}, \]
all have length $\sim 2^{j\left( 2 \widetilde{\delta} -1 \right)}$, and cover $\rho_{\theta}(B_j(\theta) \cap B(0,1))$. By the Vitali covering lemma, there is a disjoint subcollection 
\[ \{ \rho_{\theta}(T_{\theta}) : T \in \mathcal{B}_{\theta} \}, \]
indexed by some set $\mathcal{B}_{\theta}$, such that 
\[ \{ 3\rho_{\theta}(T_{\theta}) : T \in \mathcal{B}_{\theta} \}, \]
is a cover of $\rho_{\theta}(B_j(\theta) \cap B(0,1))$. The set $\mathcal{B}_{\theta}$ has cardinality $\left\lvert \mathcal{B}_{\theta}\right\rvert \leq \mu(\mathbb{R}^3)2^{j(1-\epsilon)}$; by disjointness and the definition of the ``bad'' planks. By Lemma~\ref{energylemma}, for each $j \geq 1$,
\[ \int \mu(B_j(\theta) ) \, d\lambda(\theta) \leq \int \left( \rho_{\theta\#} \mu \right) \left( \bigcup_{T_{\theta} \in \mathcal{B}_{\theta} } 3 \rho_{\theta}(T_{\theta} ) \right) \, d\lambda(\theta) \lesssim 2^{j \left(-\delta_{\epsilon}/2 + 10 \widetilde{\delta} \right)} \mu(\mathbb{R}^3). \] Since $\widetilde{\delta} \ll \delta_{\epsilon}$, summing the above inequality over $j$ gives
\[ \eqref{mainpart2} \lesssim \mu(\mathbb{R}^3). \]

It remains to show that $\rho_{\theta\#}  \mu_g \in L^2(\mathcal{H}^1)$ for $\lambda$-a.e.~$ \theta \in I$. To prove this, by Plancherel's theorem in 1 dimension it suffices to show that
\[ \int  \int_{\mathbb{R}} \left\lvert \widehat{\mu_g} \left( t \gamma(\theta) \right) \right\rvert^2 \, dt \,  d\lambda(\theta) < \infty. \]
By symmetry and by summing a geometric series, it is enough to show that for any $j \geq 1$,  
\[ \int  \int_{2^{j-1}}^{2^j}\left\lvert \widehat{\mu_g}\left( t \gamma(\theta) \right) \right\rvert^2 \, dt \,  d\lambda(\theta) \lesssim 2^{-j\epsilon}.   \]
By similar reasoning to the proof of Lemma~\ref{energylemma}, it suffices to show that
\[ \sum_{T \in \mathbb{T}_{j, g}} \int_{\mathbb{R}^3} \left\lvert M_T \mu\right\rvert^2 \lesssim 2^{j(\beta-10\epsilon) }, \]
where
\[ \mathbb{T}_{j, g} = \bigcup_{\tau \in \bigcup_{\left\lvert j' -j\right\rvert \leq 2} \Lambda_{j'}} \left\{ T \in \mathbb{T}_{\tau, g}: T \cap B\left( 0,2^{10j \widetilde{\delta}}\right) \neq \emptyset   \right\}. \]
By a similar argument to the proof of Lemma~\ref{energylemma},
\[ \sum_{T \in \mathbb{T}_{j, g}} \int_{\mathbb{R}^3} \left\lvert M_T \mu\right\rvert^2 \lesssim  2^{j \left[ \frac{5-2\alpha}{p} + \frac{1}{2} +100\epsilon \right]}. \]
Since $\alpha >4-3\beta$, $p=6$ and $\epsilon \ll \frac{4-\alpha}{3} -\beta$, this implies that
\[ \sum_{T \in \mathbb{T}_{j, g}} \int_{\mathbb{R}^3} \left\lvert M_T \mu\right\rvert^2 \lesssim  2^{j\left(\beta-10\epsilon\right)}, \]
which finishes the proof of the theorem. \end{proof}

\begin{proof}[Proof of Theorem~\ref{theoremgeneral}] By the density theorem for Hausdorff measures (\cite[Theorem~6.2]{mattila}),
\[ \limsup_{r \to 0^+} \frac{ \mathcal{H}^s(A  \cap B(x,r) )}{r^s} \leq 2^s \qquad \text{$\mathcal{H}^s$-a.e.~$x \in A$.} \] 
It follows that if, for each positive integer $n$,
\[ A_n := \left\{x \in A :  \sup_{ 0 < r < 1/n} \frac{ \mathcal{H}^s(A  \cap B(x,r) )}{r^s}  < 2^{s+1}\right\},\]  
and $\mu_n$  is the Borel measure defined by 
\[ \mu_n(F) = \mathcal{H}^s\left(F \cap A_n \setminus \bigcup_{k=1}^{n-1} A_k\right),  \]
for any Borel set $F$, then
\begin{equation} \label{muformula} \mu = \sum_{n=1}^{\infty} \mu_n, \end{equation}
and for any $n \geq 1$
\[ c_s(\mu_n) \leq \max\left\{ 2^{2s+1}, (2n)^s \mathcal{H}^s(A) \right\}. \]
By Theorem~\ref{projmeasure}, for any $n \geq 1$,
\[ \dim\left\{ \theta \in I : \rho_{\theta \#} \mu_n \not\ll \mathcal{H}^1 \right\} \leq \frac{4-s}{3}, \]
By \eqref{muformula}, 
\[ \left\{ \theta \in I : \rho_{\theta \#} \mu \not\ll \mathcal{H}^1 \right\} = \bigcup_{n=1}^{\infty} \left\{ \theta \in I : \rho_{\theta \#} \mu_n \not\ll \mathcal{H}^1 \right\}. \]
By countable stability of the Hausdorff dimension, it follows that
\[ \dim\left\{ \theta \in I : \rho_{\theta \#} \mu \not\ll \mathcal{H}^1 \right\} \leq \frac{4-s}{3}. \]
This proves the first half of the theorem. For the second half, let $B \subseteq A$ be an $\mathcal{H}^s$-measurable set with $\mathcal{H}^s(B) >0$. Let $\nu$ be the Borel measure 
\[ \nu(F) = \mathcal{H}^s(F \cap B) \qquad (= \mu(F \cap B)), \]
for any Borel set $F$. It will be shown that
\begin{equation} \label{bycontradiction} \left\{ \theta \in I: \rho_{\theta \#} \nu \ll \mathcal{H}^1 \right\}  \subseteq \left\{ \theta \in I: \mathcal{H}^1( \rho_{\theta} (B) ) > 0 \right\}. \end{equation}
Let $\theta \in I$ be such that $\rho_{\theta \#} \nu \ll \mathcal{H}^1$. Suppose for a contradiction that $\mathcal{H}^1(\rho_{\theta}(B)) = 0$. Let $\delta>0$ be such that 
\[ (\rho_{\theta\#} \nu)(F) < \mathcal{H}^s(B), \]
for any Borel set $F$ with $\mathcal{H}^1(F) < \delta$. Since $\mathcal{H}^1(\rho_{\theta}(B)) = 0$, there exists a Borel set $F$ containing $\rho_{\theta}(B)$ with $\mathcal{H}^1(F) < \delta$. Hence 
\[ \mathcal{H}^s(B) = \mathcal{H}^s( \rho_{\theta}^{-1}(F) \cap B ) = (\rho_{\theta\#} \nu)(F) < \mathcal{H}^s(B), \]
and this contradiction proves \eqref{bycontradiction}. Thus
\begin{align*} \left\{ \theta \in I: \mathcal{H}^1( \rho_{\theta} (B) ) = 0 \right\} &\subseteq \left\{ \theta \in I: \rho_{\theta \#} \nu \not\ll \mathcal{H}^1 \right\} \\
&\subseteq \left\{ \theta \in I: \rho_{\theta \#} \mu \not\ll \mathcal{H}^1 \right\}. \qedhere \end{align*} 
  \end{proof}

\appendix
\section{Refined decoupling}

The following inequality is Theorem 9 from \cite{GGGHMW}; it is a refined version of the decoupling theorem for generalised cones. 
\begin{theorem}[{\cite[Theorem~9]{GGGHMW}}] \label{refineddecouplingtheorem} Let $I$ be a compact interval, and let $\gamma :I \to S^2$ be a $C^2$ unit speed curve with $\det(\gamma, \gamma', \gamma'' )$ nonvanishing on $I$. Then if $c>0$ is sufficiently small (depending only on $\gamma$), then for any $\epsilon >0$, there exists $\delta_0>0$ such that the following holds for all $0 < \delta < \delta_0$, and any $R \geq 1$. Let $\Theta_R$ be a maximal $c R^{-1/2}$-separated subset of $I$, and for each $\theta \in \Theta_R$, let
\begin{multline*} \tau(\theta) := \\
\left\{ \lambda_1 \gamma(\theta) + \lambda_2 \gamma'(\theta) + \lambda_3 (\gamma \times \gamma')(\theta) : 1/2 \leq \lambda_1 \leq 1, \lvert \lambda_2 \rvert \leq R^{-1/2}, \lvert \lambda_3 \rvert \leq R^{-1} \right\}. \end{multline*}
For each $\tau = \tau(\theta)$, let $\mathbb{T}_{\tau}$ be a  $\sim 1$-overlapping cover of $\mathbb{R}^3$ by translates of 
\[ \left\{ \lambda_1 \gamma(\theta) + \lambda_2 \gamma'(\theta) + \lambda_3 (\gamma \times \gamma')(\theta) : \lvert \lambda_1 \rvert \leq R^{\delta}, \lvert \lambda_2 \rvert \leq R^{1/2 + \delta}, \lvert \lambda_3 \rvert \leq R^{1+\delta} \right\}. \]
If $2 \leq p \leq 6$, and 
\[ \mathbb{W} \subseteq \bigcup_{\theta \in \Theta_R} \mathbb{T}_{\tau(\theta)}, \]
and 
\[ \sum_{T \in \mathbb{W} } f_T \]
is such that $\lVert f_T \rVert_p$ is constant over $T \in \mathbb{W}$ up to a factor of 2, with $\supp \widehat{f_T} \subseteq \tau(T)$ and 
\[ \lVert f_T\rVert_{L^{\infty}(B(0,R) \setminus T) } \leq A R^{-10000} \lVert f_T \rVert_p, \]
and $Y$ is a disjoint union of balls in $B_3(0,R)$ of radius 1, such that each ball $Q \subseteq Y$ intersects at most $M$ planks $2T$ with $T \in \mathbb{W}$, then 
\[ \left\lVert \sum_{T \in \mathbb{W} } f_T \right\rVert_{L^p(Y) } \leq C_{A, \gamma,c, \epsilon, \delta} R^{\epsilon} \left( \frac{M}{\left\lvert \mathbb{W} \right\rvert } \right)^{\frac{1}{2} - \frac{1}{p} } \left( \sum_{T \in \mathbb{W} } \left\lVert f_T \right\rVert_p^2 \right)^{1/2}. \] \end{theorem}


\begin{thebibliography}{}

\bibitem{davies}
Davies, R. O.: Subsets of finite measure in analytic sets.
\newblock Nederl. Akad. Wetensch. Proc. Ser. A. \textbf{55} = Indagationes Math. \textbf{14}, 488--489 (1952)

\bibitem{fasslerorponen14}
Fässler,~K., Orponen,~T.: On restricted families of projections in $\mathbb{R}^3$.
\newblock Proc.~London Math.~Soc. (3) \textbf{109}, 353--381 (2014)

\bibitem{folland}
Folland, G.~B.: Real analysis. Modern techniques and their applications. Second edition.
\newblock Pure and Applied Mathematics (New York). A Wiley-Interscience Publication. John Wiley \& Sons, Inc., New York, (1999)

\bibitem{GGGHMW}
Gan,~S., Guo,~S., Guth,~L., Harris,~T.~L.~J., Maldague,~D., Wang,~H.: On restricted projections to planes in $\mathbb{R}^3$.
\newblock arXiv:2207.13844v1 (2022)

\bibitem{ggm}
Gan,~S., Guth,~L., Maldague,~D.: An exceptional set estimate for restricted projections to lines in $\mathbb{R}^3$.
\newblock arXiv:2209.15152v2 (2022)

\bibitem{GIOW}
Guth,~L., Iosevich,~A., Ou,~Y, Wang,~H.: On Falconer’s distance set problem in the plane.
\newblock Invent.~Math. \textbf{219}, 779--830 (2020)

\bibitem{KOV}
Käenmäki,~A., Orponen,~T., Venieri,~L.: A Marstrand-type restricted projection theorem in $\mathbb{R}^3$. To appear in \emph{Amer.~J.~Math.}
\newblock arXiv:1708.04859v2 (2017)

\bibitem{liu}
Liu,~B.: Hausdorff dimension of pinned distance sets and the $L^2$-method.
\newblock Proc.~Amer.~Math.~Soc. \textbf{148} 333--341 (2020)

\bibitem{marstrand}
Marstrand,~J.: Some fundamental geometrical properties of plane sets of fractional
dimensions.
\newblock Proc.~Lond.~Math.~Soc. (3) \textbf{4}, 257--302 (1954)

\bibitem{mattila}
Mattila,~P.: Geometry of sets and measures in Euclidean spaces.
\newblock Cambridge University Press, Cambridge, United Kingdom (1995)

\bibitem{pramanik}
Pramanik,~M., Yang,~T., Zahl,~J.: A Furstenberg-type problem for circles, and a Kaufman-type restricted projection theorem in $\mathbb{R}^3$.
\newblock arXiv:2207.02259v2 (2022)

\bibitem{tao}
Tao, T.: An introduction to measure theory.
\newblock Graduate Studies in Mathematics, \textbf{126}. American Mathematical Society, Providence, RI, (2011)

\end{thebibliography}
\end{document}